\documentclass[12pt]{amsart}

\usepackage{amsmath,amsthm,amscd,amsfonts,amssymb,epic,eepic,bbm}
\usepackage[pagebackref,colorlinks=true,linkcolor=blue,citecolor=blue]{hyperref}

\allowdisplaybreaks

\newtheorem{thm}{Theorem}[section]
\newtheorem{cor}[thm]{Corollary}

\newtheorem{lem}[thm]{Lemma}

\newtheorem{prop}[thm]{Proposition}
\theoremstyle{remark}
\newtheorem*{rem}{Remark}

\newcounter{remarkscounter}
\newenvironment{remarks}
{\medskip\noindent{\it
Remarks.}\begin{list}{{\rm(\arabic{remarkscounter})}
}{\usecounter{remarkscounter}

\setlength{\labelsep}{\fill} \setlength{\leftmargin}{0pt}
\setlength{\itemindent}{\fill}
\setlength{\labelwidth}{\fill}\setlength{\topsep}{0pt}
\setlength{\listparindent}{0pt}}} {\end{list}}

\numberwithin{equation}{section}
\newcommand{\A}{\mathbb{A}}
\newcommand{\GL}{\mathrm{GL}}

\newcommand{\ZZ}{\mathbb{Z}}

\newcommand{\Gal}{\mathrm{Gal}}

\newcommand{\QQ}{\mathbb{Q}}

\newcommand{\lto}{\longrightarrow}

\newcommand{\OO}{\mathcal{O}}
\newcommand{\CC}{\mathbb{C}}
\newcommand{\RR}{\mathbb{R}}

\newcommand{\quash}[1]{}

\theoremstyle{definition}
\newtheorem{defn}[thm]{Definition}

\newcommand{\rtr}{\mathrm{rtr}}

\renewcommand{\bar}{\overline}

\numberwithin{equation}{subsection}

\begin{document}
\title[General simple relative trace formula]{A general simple relative trace formula}
\author{Jayce R. Getz}
\author{Heekyoung Hahn}
\address{Department of Mathematics, Duke University, Durham, NC 27708}
\email{jgetz@math.duke.edu}
\address{Department of Mathematics, Duke University, Durham, NC 27708}
\email{hahn@math.duke.edu}

\subjclass[2010]{Primary 11F70\,;  Secondary 35P20}

\thanks{The first author is thankful for partial support provided by NSF grants DMS 1405708.  Any opinions, findings, and conclusions or recommendations expressed in this material are those of the author and do not necessarily reflect the views of the National Science Foundation.}

\begin{abstract}
In this paper we prove a relative trace formula for all pairs of connected algebraic groups $H \leq G \times G$ with $G$ a reductive group and $H$ the direct product of a reductive group and a unipotent group given that the test function satisfies simplifying hypotheses.
As an application, we prove a relative analogue of the Weyl law, giving an asymptotic formula for the number of eigenfunctions of the Laplacian on a locally symmetric space associated to $G$ weighted by their $L^2$-restriction norm over a locally symmetric subspace associated to $H_0 \leq G$.
\end{abstract}

\maketitle

\tableofcontents

\section{Introduction}\label{intro}

Let $G$ be a connected reductive algebraic group over a number field $F$ and let $A_G$ be the neutral component of the real points of the greatest $\QQ$-split torus in the center of $\mathrm{Res}_{F/\QQ}G$. Throughout this paper, we let 
$$
H \leq G \times G
$$
be a connected algebraic subgroup such that $H$ is the direct product of a reductive group and a unipotent group; both of these groups are necessarily connected.  We do not assume that the decomposition of $H$ into a reductive and unipotent group is compatible with the embedding $H \hookrightarrow G\times G$.

Let $\chi:H(\A_F) \to \CC^\times$ be  a quasi-character trivial on $A_{G, H}H(F)$ (see \S \ref{ssec-A} for the definition of $A_{G, H}$ and the other $A_?$ groups; they are all central subgroups).  Let
 $$
 \phi \in  L^2_{\mathrm{cusp}}(A_GG(F) \backslash G(\A_F) \times A_GG(F) \backslash G(\A_F))
 $$ be a smooth cusp form, and let
\begin{align} \label{period}
\mathcal{P}_{\chi}(\phi):=\int_{A_{G, H}H(F)\backslash H(\A_F)} \chi(h_\ell,h_r)\phi(h_\ell,h_r) d(h_\ell,h_r)
\end{align}
whenever this period is well-defined (for a criterion see Corollary \ref{cor-autom-form} below).  Here $d(h_{\ell},h_r)$ is a Haar measure; we will set our conventions on Haar measures in \S \ref{ssec-Haar} below.
The relative trace formula is a tool for studying the period integrals $\mathcal{P}_{\chi}(\phi)$.  
 Many particular instances of the relative trace formula have been developed, but the development has not been systematic. 

 In this paper we establish the formula in what we view is the natural level of generality in terms of the subgroup $H$ for test functions satisfying the usual ``simple trace formulae'' hypotheses.  In particular, we only make the assumption that $H$ is connected and a direct product of a reductive and unipotent group.  In contrast, in all references known to the authors the subgroup $H$ is assumed to be ``large'' e.g.~spherical and satisfy other simplifying hypotheses.  We also note that this greater generality is not vacuous in that it leads to new applications, for example, Theorem \ref{thm-Weyl-intro} below.  It is also used in constructing the four-variable automorphic kernel functions of \cite{Getz4}.
 
For $f \in C_c^\infty(A_G\backslash G(\A_F))$ let
\begin{align*}
R(f):L^2(A_G G(F) \backslash G(\A_F)) &\lto L^2(A_G G(F) \backslash G(\A_F))\\
\varphi &\longmapsto \left(x \mapsto \int_{A_G \backslash G(\A_F)}f(g)\varphi(xg)dg\right)
\end{align*}
denote the operator defined by the right regular action and $f$. We prove the following theorem:
\begin{thm}\label{thm-stf}
Let $f  \in C_c^\infty(A_G\backslash G(\A_F))$ be a function such that $R(f)$ has cuspidal image and such that if the $H(\A_F)$-orbit of $\gamma \in G(F)$ intersects the support of $f$ then $\gamma$ is elliptic, unimodular and closed.
Then
$$
 \sum_\gamma \tau(H_\gamma)\mathrm{RO}^{\chi}_\gamma(f) = \sum_\pi \rtr\,\pi(f)
$$
where the sum on $\gamma$ is over elliptic unimodular closed relevant classes and the sum on $\pi$ is over isomorphism classes of cuspidal automorphic representations of $A_G \backslash G(\A_F)$.
\end{thm}

\noindent Here elliptic, unimodular and closed are defined as in \S \ref{ssec-rel-classes}, the action of $H$ on $G$ is given in \eqref{action}, and relevant is defined as in \S \ref{ssec-roi}.  Moreover, 
$\tau(H_{\gamma})$ is a volume term that can be viewed as a Tamagawa number if normalized appropriately, $\mathrm{RO}_{\gamma}^\chi(f)$ is a relative orbital integral (see \S \ref{sec-geo-side} for both of these notions), and $\rtr\,\pi(f)$ is the relative trace of $\pi(f)$, defined in \eqref{rtr} (it is a period integral of the form \eqref{period}).  Moreover a cuspidal automorphic representation $\pi$ of $A_G \backslash G(\A_F)$, by convention, is an automorphic representation of $G(\A_F)$ trivial on $A_G$ that can be realized in $L^2_{\mathrm{cusp}}(A_GG(F) \backslash G(\A_F))$.  In particular, we do not fix an embedding; the definition of $\rtr\,\pi(f)$ involves the entire $\pi$-isotypic subspace of $L^2_\mathrm{cusp}(A_G G(F) \backslash G(\A_F))$.

\begin{remarks} 

\item Given work of Lindenstrauss and Venkatesh \cite{LV}, the assumption that $R(f)$ has purely cuspidal image may not be as severe a restriction as one might think (see also the proof of Theorem \ref{thm-Weyl-main}). 

\item Though the method of proof is the usual one (take a kernel and compute the integral over $A_{G,H}H(F) \backslash H(\A_F)$ two ways) there are many points in the proof of Theorem \ref{thm-stf} that are not obvious.  
On the spectral side we check that $\mathrm{rtr}\,\pi(f)$ is well-defined for all $f$, not just $K_\infty$-finite $f$.  On the geometric side we define a notion of elliptic elements and the relative analogue of semisimple elements (which we call unimodular and closed).  These have only appeared in special cases in the literature.  We also use Galois cohomology to deal with non-connected stabilizers in a way that we have never seen in the literature in the context of the relative trace formula.
\end{remarks}

The formula in Theorem \ref{thm-stf} is called \textbf{simple} because we have imposed conditions on the test function $f$ to ensure that various analytic difficulties disappear. Theorem \ref{thm-stf} is \textbf{general} because
the geometric set-up includes all trace formulae that the authors have seen as special cases.  
For example, the simple twisted relative trace formula of the second author \cite{Hahn} is a special case of this formula as well as is the usual simple trace formula of Deligne and Kazhdan  \cite{BDKV} (see also \cite{Ro}), as one can see by taking $\chi$ to be trivial and $H$ to be the diagonal copy of $G$ inside $G \times G$.  As another example, let $E/F$ be a quadratic extension, let $G=\mathrm{Res}_{E/F}\GL_n$, let $U_n\leq G$ be a unitary group, let $N \leq G$ the unipotent radical of the Borel subgroup of upper triangular matrices, let $\psi:N(F) \backslash N(\A_F) \lto \CC^\times$ be a character, and set 
$$
H=U_n \times N \quad \textrm{ and } \quad \chi=1 \times \psi.
$$
In this case the trace formula above is a simple version of one introduced by Jacquet and Ye (compare \cite{JY}).  We also note that the formula does not hold for a general connected algebraic subgroup $H \leq G \times G$ without serious modification (see the remark after Proposition \ref{agr}), so in some sense it is as general as possible.  

As an application of these ideas, we prove a relative analogue of the Weyl law in Theorem \ref{thm-Weyl-intro} below.  It gives an asymptotic formula for the number of eigenfunctions of the Laplacian on a locally symmetric space associated to $G$ weighted by the $L^2$-restriction norm over a locally symmetric subspace associated to $H_0\leq G$. 

To state it, assume that $G$ is split and adjoint over $\QQ$. Note that  $G(\QQ) \backslash G(\A_{\QQ})$ is of finite volume but non-compact. Let $H_0 \leq G$ be the direct product of a reductive group and a unipotent group and
$$K:=K_\infty \times K^{\infty}\leq G(\A_\QQ)
$$
where $K_\infty \leq G(\RR)$ is a maximal compact subgroup and $K^\infty \leq G(\A_\QQ^\infty)$ is a compact open subgroup satisfying the torsion-freeness assumption (TF) of \S \ref{sec-Weyl} below. 

In the setting above, using a technique developed by Lindenstrauss and Venkatesh \cite{LV}, we prove Theorem \ref{thm-Weyl-intro} below.  We remark that since $G(\QQ)\backslash G(\A_\QQ)$ is non-compact, even if $H_0(\QQ)\backslash H_0(\A_\QQ)$ is compact the theorem does not
follow in any obvious way from the classical Weyl law or its local variants.

\begin{thm}\label{thm-Weyl-intro}   Assume that $H_0(\QQ) \backslash H_0(\A_{\QQ})$ is compact.
As $X \to \infty$
one has
\begin{align*} 
\sum_{\pi:\,\pi(\Delta) \leq X}\sum_{\varphi \in \mathcal{B}(\pi)^K}\int_{H_0(\QQ) \backslash H_0(\A_{\QQ})}|\varphi(h)|^2dh \sim \alpha(G)\mathrm{meas}_{dh}(H_0(\QQ) \backslash H_0(\A_{\QQ}))X^{d/2},
\end{align*}
where the sum is over isomorphism classes of cuspidal automorphic representations  $\pi$ of $G(\A_\QQ)$, $\mathcal{B}(\pi)$ is an orthonormal basis of $\pi$-isotypic subspace of $L^2_{\mathrm{cusp}}(G(\QQ) \backslash G(\A_{\QQ}))$, $\pi(\Delta)$ is the eigenvalue of the Casimir operator $\Delta$ acting on the space of $K_\infty$-fixed vectors in $\pi$,  $\alpha(G)>0$ is a constant related to the Plancherel measure defined in \cite{LV}, and $d=\dim (G(\RR)/K_{\infty})$.  
\end{thm}
We refer to the asymptotic in Theorem \ref{thm-Weyl-intro} as a relative Weyl law.  
We can in fact weaken the assumption that $H_0(\QQ) \backslash H_0(\A_\QQ)$ is compact.  Specifically, in Proposition \ref{upper} we prove that if $H_0(\QQ)\backslash H_0(\A_\QQ)$ is of finite volume but non-compact, then the relative Weyl law holds provided that one assumes the upper bound of the relative Weyl law (in the setting of the usual Weyl law this was proven in \cite{Donnelly}).  Interestingly, this is not known in the relative case.

We now outline the sections of this paper.  In the following section we recall the notion of relative classes and relative analogues of definitions often used in the context of the absolute trace formula.  
The proof of Theorem \ref{thm-stf} comes down to evaluating an integral of a kernel function in two ways.  The spectral evaluation is given in \S \ref{sec-rel-tr} and the geometric evaluation is given in \S \ref{sec-geo-side}. Finally, in \S \ref{sec-Weyl} we prove Theorem \ref{thm-Weyl-intro}.

\section*{Acknowledgements}

The authors thank M.~Stern answering questions on the Weyl law in the context of differential geometry.  We also thank the referee for remarks that improved the exposition.

\section{Preliminaries and notation}

\subsection{Relative classes} \label{ssec-rel-classes}

Let $G$ be a connected reductive algebraic group over a characteristic zero field $F$ with algebraic closure $\bar{F}$ and let 
$$
H \leq G \times G
$$ 
be a connected algebraic subgroup that is the direct product of a reductive and a unipotent group.  We let
$$
\mathrm{diag}:G \lto G \times G
$$
denote the diagonal embedding. The letter $R$ will denote an $F$-algebra.
There is an action of $H$ on $G$ given at the level of points by 
\begin{align} \label{action}
\cdot:H(R) \times G(R) &\lto G(R)\\
((h_\ell,h_r),g) &\longmapsto h_\ell gh_r^{-1}. \nonumber
\end{align}
The stabilizer of a $\gamma \in G(R)$ will be denoted by $H_\gamma$.  
By assumption, we can write
$$
H=H^r \times H^u
$$
where $H^r$ is reductive and $H^u$ is unipotent.

\begin{defn} Let $k/F$ be a field.  An element $\gamma \in G(k)$ is 
\begin{itemize}
\item \textbf{closed} if the orbits of $\gamma$ under $H$ and $H^r$ are both closed.
\item \textbf{unimodular} if $H_{\gamma}$ is the direct product of a reductive and a unipotent group.
\item \textbf{elliptic} if the maximal reductive quotient of $H_{\gamma}/\mathrm{diag}(Z_G) \cap H$ has anisotropic center.
\end{itemize}
\end{defn}

\begin{rem} If $H$ is reductive, then a closed element has reductive stabilizer and hence is unimodular.
\end{rem}

If $R$ is an $F$-algebra, then an element of 
\begin{align}
\Gamma(R):=H(R) \backslash G(R)
\end{align}
is called a \textbf{relative class}, or simply a class.  
Note that here the quotient is taken with respect to the action \eqref{action}.  All of the conditions mentioned in the previous definition depend only on the relative class of an element of $\Gamma(R)$, and not on the particular element. 
If $k$ is a field with algebraic closure $\bar{k}$ we say that $\gamma,\gamma' \in G(k)$ are in the same \textbf{geometric class} if there is an $h \in H(\bar{k})$ such that $h \cdot \gamma=\gamma'$.  We denote the set of geometric classes by
\begin{align} \label{geo}
\Gamma^{\mathrm{geo}}(k):=\mathrm{Im}(G(k) \to H \backslash G(k)).
\end{align}

\subsection{The $A$ groups} \label{ssec-A}

If $H$ is a connected algebraic group over a number field $F$ we let $A_H$ be the neutral component (in the real topology) of the real points of the maximal $\QQ$-split torus in $\mathrm{Res}_{F/\QQ}H$.  We let
\begin{align}
A_{G,H}:=A_H \cap (A_G \times A_G)\\
A:=A_H \cap \mathrm{diag}(A_G). \nonumber
\end{align}
We choose Haar measures $da_G$ on $A_G$, $d(a_\ell,a_r)$ on $A_{G,H}$ and $da$ on $A$.

\subsection{Haar measures} \label{ssec-Haar}

Throughout this work we fix a Haar measure $dg$ on $G(\A_F)$ and use it and $da$ to obtain a Haar measure, also denoted by $dg$, on $A_G \backslash G(\A_F)$.  We also fix a Haar measure $d(h_\ell,h_r)$ on $H(\A_F)$ and also denote by $d(h_\ell, h_r)$ the induced measure on $A_{G,H} \backslash H(\A_F)$.
For each unimodular $\gamma \in H(F)$ we let $d(h_\ell,h_r)_\gamma$ be a Haar measure on $H_{\gamma}(\A_F)$
 and let 
$$
\dot{d}(h_{\ell},h_r)
$$
denote the induced right-invariant Radon measure on $H_\gamma(\A_F)\backslash H(\A_F)$.

\section{Relative traces} \label{sec-rel-tr}

As in the introduction, let 
$$
\chi:H(\A_F) \lto \CC^\times
$$
be a quasi-character trivial on $A_{G,H}H(F)$.
Let $f \in C_c^\infty( A_G\backslash G(\A_F))$, and let $\pi$ be a cuspidal automorphic representation of $A_G \backslash G(\A_F)$. 
We let $\mathcal{B}(\pi)$ be an orthonormal basis of the $\pi$-isotypic subspace of 
$L^2_{\mathrm{cusp}}(A_GG(F) \backslash G(\A_F))$ consisting of smooth vectors and let
\begin{align}
K_{\pi(f)}(x,y):=\sum_{\varphi \in \mathcal{B}(\pi)}R(f)\varphi(x)\overline{\varphi}(y).
\end{align}
A priori this expression only converges in $L^2(A_GG(F) \backslash G(\A_F) \times A_GG(F) \backslash G(\A_F))$.  However, it follows from the Dixmier-Malliavin lemma \cite{DM} that there is a unique smooth (jointly in $(x,y)$) square-integrable function that represents $K_{\pi(f)}$ (compare the proof of Theorem \ref{thm-spectral-side}).  From now on we use the notation $K_{\pi(f)}$ to refer to this function, and whenever $R(f)$ has cuspidal image we let
\begin{align} \label{cusp-ker}
K_{f}(x,y):&=\sum_{\pi} \sum_{\varphi \in \mathcal{B}(\pi)} R(f)\varphi(x)\bar{\varphi}(y)
\end{align}
where the sum is over isomorphism classes of cuspidal automorphic representations $\pi$ of $A_G \backslash G(\A_F)$.

We refer to the integral 
\begin{align} \label{rtr}
\rtr\,\pi(f):=\rtr_{H,\chi}(\pi(f)):=\mathcal{P}_{\chi}(K_{\pi(f)})
\end{align}
as the \textbf{relative trace} of $\pi(f)$, where $\mathcal{P}_{\chi}$ is the period integral defined in  \eqref{period} above.  We will show in the course of the proof of Theorem \ref{thm-spectral-side} that the integral in the definition of $\mathcal{P}_{\chi}(K_{\pi(f)})$ is well-defined.

The following theorem amounts to the computation of the spectral side of our relative trace formula:

\begin{thm} \label{thm-spectral-side}
Let $f\in C_c^\infty(A_G\backslash G(\A_F))$, and assume that $R(f)$ has cuspidal
image. Then
$$
\int_{A_{G,H}H(F)\backslash H(\A_F)}\chi(h_\ell,h_r)
K_{f}(h_\ell,h_r)d(h_\ell,h_r)=\sum_\pi \rtr\,\pi(f).
$$
Moreover, the integral on the left and the sum on the right are absolutely convergent.
\end{thm}

\noindent This is the main result of this section. A similar result is proven in 
\cite{Hahn} in a special case, but we give a simpler proof here.  

Fix a maximal compact subgroup $K_\infty$ of $G(F_\infty)$, where $F_\infty:=\prod_{v|\infty}F_v$ is the product of the archimedian completions of $F$.
As mentioned above, in the course of the proof of theorem we will prove that the integral in the definition of $\rtr\,\pi(f)$ is absolutely convergent.  Assuming this for the moment, we obtain the following corollary:

\begin{cor} \label{cor-autom-form} Assume that $\varphi \in L^2_{\mathrm{cusp}}(A_GG(F) \backslash G(\A_F))$ is a cuspidal automorphic form, that is, $\varphi$ is cuspidal, $K_\infty$-finite, and finite under the center of the universal enveloping algebra of $\mathrm{Lie}(\mathrm{Res}_{F/\QQ}G(\RR)) \otimes_\RR \CC$.  Then the integral defining $\mathcal{P}_{\chi}(\varphi \times \overline{\varphi})$ is absolutely convergent.  
\end{cor}

\begin{proof} It suffices to verify the corollary when $\varphi$ lies in the $\pi$-isotypic subspace $L^2_{\mathrm{cusp}}(\pi)$ 
of the cuspidal subspace of
$L^2(A_G G(F)\backslash G(\A_F))$ for a cuspidal automorphic representation $\pi$.
By a standard argument one can choose an $f \in C_c^\infty(A_G\backslash G(\A_F))$ such that $R(f)\varphi=\varphi$
and $R(f)$ acts by zero on the orthogonal complement of $\varphi$ in $L^2_{\mathrm{cusp}}(\pi)$.  Hence
$$
\mathcal{P}_{\chi}(\varphi \times \overline{\varphi})= \mathcal{P}_{\chi}\left(K_{\pi(f)} \right)=\rtr\,\pi(f).
$$
\end{proof}

\subsection{Integrals of rapidly decreasing functions}

Let $Z \leq \mathrm{Res}_{F/\QQ}G$ be the maximal split torus in the center of $G$.  Let $T \leq \mathrm{Res}_{F/\QQ}G \times \mathrm{Res}_{F/\QQ}G$ be a maximal split torus and let $\Delta$ be a choice of simple roots of $T/(Z \times Z)$ in $\mathrm{Res}_{F/\QQ}G \times \mathrm{Res}_{F/\QQ}G$.
Set
$$
A^G:=T(\RR)^{+}/A_G \times A_G
$$
where the $+$ denotes the neutral component in the real topology.  For any positive real number $r$ we set
\begin{align}
A_r^G:&=\{t \in A^G:t^{\alpha}>r \textrm{ for all } \alpha \in \Delta\}.
\end{align}
For concreteness, we record the following definition:
\begin{defn} \label{defn-rd} A function
$$
\phi: A_GG(F) \backslash G(\A_F)  \times A_GG(F)\backslash G(\A_F) \lto \CC
$$
is \textbf{rapidly decreasing} if it is smooth and for all compact subsets $\Omega \subset (A_GG(F) \backslash G(\A_F))^{\times 2}$, $r \in \RR_{>0}$, and $p \in \ZZ$ there is a constant $C=C_{\Omega,r,p}$ such that one has
$$
|\phi(tx)|\leq C t^{\alpha p}
$$
for all $t \in A_r^G$, $x \in \Omega$, and $\alpha \in \Delta$.
\end{defn}

\noindent To ease confusion, above and below we use the symbol $\phi$ for a function on $(A_G G(F)\backslash G(\A_F))^{ \times 2}$ and $\varphi$ for a function on $A_G G(F) \backslash G(\A_F)$.

\begin{prop}\label{agr}
For all rapidly decreasing (smooth) functions $\phi\in
L^2((A_GG(F) \backslash G(\A_F))^{ \times 2})$, the period integral
$$
\mathcal{P}_\chi(\phi):=\int_{A_{G,H}H(F)\backslash
H(\A_F)}\chi(h_\ell,h_r)\phi(h_\ell,h_r)d(h_\ell,h_r)
$$
is absolutely convergent.  
\end{prop}

\begin{proof} Since $H$ is the direct product of a unipotent group and a reductive group, and $U(F) \backslash U(\A_F)$ is compact for any unipotent group $U$, it suffices to prove the proposition in the special case where $H$ is reductive.  In this case, the argument proving \cite[Proposition 1]{AGR} implies the proposition.
\end{proof}

\begin{rem} This proposition depends crucially on the fact that $H$ is assumed to be a direct, not a semidirect, product of a reductive group and a unipotent group.  It is false for a general connected algebraic group.  Examples of this occur already in low-rank applications of the Rankin-Selberg method (see \cite[Lemma 10.3]{GG} for an example).
\end{rem}

We also recall the following basic theorem \cite[(15)']{godement_cusp}:
\begin{thm}[\cite{godement_cusp}]\label{thm-godement} 
Let $r \in \RR_{>0}$, $p \in \ZZ$ and let $\Omega \subset (A_GG(F) \backslash G(\A_F))^{ \times 2}$ be compact.  If $\Phi \in C_c^\infty((A_G \backslash G(\A_F))^{ \times 2})$ then one has an estimate
$$
|R(\Phi)\phi(tx)|\leq C  t^{\alpha p}\parallel\phi\parallel
$$
for all $\phi \in L^2_{\mathrm{cusp}}((A_GG(F) \backslash G(\A_F))^{\times 2})$,  $t \in A_r^G$, $\alpha \in \Delta$, and $x \in \Omega$, where the constant $C:=C_{r,p,\Omega,\Phi}$ is independent of $\phi$.  In particular, $R(\Phi)\phi$ is rapidly decreasing.  
\qed
\end{thm}

\subsection{Proof of Theorem \ref{thm-spectral-side}}

 By assumption, $R(f)$ has image in the cuspidal spectrum.  Thus the operator $R(f)$ is trace class and hence is Hilbert-Schmidt.
 We therefore have the convergent
$L^2$-expansion
\begin{align} \label{ker-above}
K_{f}(x,y)&= \sum_{\pi} K_{\pi(f)}(x,y)=
\sum_{\pi} \sum_{\varphi \in \mathcal{B}(\pi)} R(f)\varphi(x)\bar{\varphi}(y)
\end{align}
where the sum is over isomorphism classes of cuspidal automorphic representations of $A_G \backslash G(\A_F)$.  
By the Dixmier-Malliavin lemma \cite{DM} we can write $f$ as a finite sum of functions of the form 
$$
f_1*f_2*f_3 
$$
for $f_1,f_2,f_3 \in C_c^\infty(A_G \backslash G(\A_F))$.  It clearly suffices to prove the theorem for $f$ of this special form, so for the moment we assume that $f=f_1*f_2*f_3$.  For $f \in C_c^\infty(A_G \backslash G(\A_F))$ let
$$
(f)^{\vee}(g):=f(g^{-1}).
$$
We note that
\begin{align*}
\sum_{\varphi \in \mathcal{B}(\pi)} R(f)\varphi(x)\bar{\varphi}(y)
&=\sum_{\varphi \in \mathcal{B}(\pi)} \varphi(x)R((f)^\vee)\bar{\varphi}(y)
\end{align*}
because they both represent the same kernel.
Thus 
\begin{align} \label{is-rap-dec}
K_{\pi(f)}(x,y)&=\sum_{\varphi \in \mathcal{B}(\pi)} R(f_1*f_2*f_3)\varphi(x)\bar{\varphi}(y)=\sum_{\varphi \in \mathcal{B}(\pi)} R(f_2*f_3)\varphi(x)R(f_1^{\vee})\bar{\varphi}(y) \\
&=  (R(f_2) \times R(f_1^{\vee}))\sum_{\varphi \in \mathcal{B}(\pi)} R(f_3)\varphi(x)\bar{\varphi}(y). \nonumber
\end{align}
The latter function is smooth as a function of $(x,y)$ (jointly) and this is the unique smooth function representing  $K_{\pi(f)}(x,y)$ as mentioned earlier (to prove convergence one can invoke Theorem \ref{thm-godement}).  Thus we can view $K_{\pi(f)}(x,y)$ as an honest function.    The same is true of $K_{f}(x,y)$ and \eqref{ker-above} holds pointwise.  

Thus in view of Proposition \ref{agr}, to complete the proof of the theorem it suffices to show that for any $f \in C_c^\infty(A_G\backslash G(\A_F))$ one has that
\begin{align} \label{is-rd2}
\sum_{\pi} |K_{\pi(f)}(x,y)|
\end{align}
is rapidly decreasing as a function of $(x,y) \in (A_G G(F)\backslash G(\A_F))^{\times 2}$.  
To see this, we use a trick going back to Selberg.  
Using the Dixmier-Malliavin lemma we reduce to the case where $f=f_1*f_2$.  For $f \in C_c^\infty(A_G \backslash G(\A_F))$ we set $f^*(g):=\bar{f(g^{-1})}$.  
Applying the Cauchy-Schwarz inequality
 we obtain
\begin{align*}
|K_{\pi(f)}(x,y)|^2 &=\left|\sum_{\varphi \in \mathcal{B}(\pi)}
\pi(f_1)\varphi(x)
\overline{\pi(f_2^*)\varphi}(y) \right|^2 \\
& \leq \sum_{\varphi \in \mathcal{B}(\pi)}\left|\pi(f_1)\varphi(x)\right|^2\sum_{\varphi \in \mathcal{B}(\pi)}\left|\bar{\pi(f_2^*)\varphi}(y)\right|^2\\
&=K_{\pi(f_1^**f_1)}(x,x) K_{\pi(f_2*f_2^*)}(y,y).
\end{align*}
We note that originally the first identity is an identity of $L^2$-functions, but using the Dixmier-Malliavin lemma and Theorem \ref{thm-godement} as above we can regard it as a pointwise identity of continuous functions.  The same is true of the rest of the functions appearing in the inequalities above, and in particular the application of Cauchy-Schwarz makes sense.  The point of all of this is that the kernels
$K_{\pi(f_1*f_1^*)}(x,x)$, $ K_{\pi(f_2^**f_2)}(y,y)$ are positive as functions of $x$ and $y$.  

By H\"older's inequality one has
$$
\sum_{\pi}(K_{\pi(f_1^**f_1)}(x,x)K_{\pi(f_2*f_2^*)}(y,y))^{1/2} \leq \left(\sum_{\pi}K_{\pi(f_1^**f_1)}(x,x)\right)^{1/2} \left(\sum_{\pi} K_{\pi(f_2*f_2^*)}(y,y) \right)^{1/2}.
$$
Thus it is enough to prove that for all $h \in C_c^\infty(A_G \backslash G(\A_F))$ the sum
\begin{align} \label{is-rd3}
\sum_{\pi}K_{\pi(h)}(x,x)
\end{align}
is rapidly decreasing as a function of $x$.  Using the Dixmier-Malliavin lemma again we reduce to the case that $h=h_1*h_2*h_3$, and 
arguing as in the beginning of the proof we obtain 
\begin{align}
\sum_\pi K_{\pi(h)}(x,y)=R(h_2) \times R(h_1^{\vee}) \sum_{\pi}K_{\pi(h_3)}(x,y).
\end{align}
In the notation of Definition \ref{defn-rd}, Theorem \ref{thm-godement} implies that for all compact subsets 
$\Omega \subset (A_GG(F) \backslash G(\A_F))^{\times 2}$, $x \in \Omega$, $r \in \RR_{>0}$, and $p \in \ZZ$ one has
$$
\left|\sum_{\pi} K_{\pi(h)}(tx,tx) \right|
 \ll_{h_1,h_2,\Omega,r,p} t^{\alpha p} \left(\sum_{\pi}\mathrm{tr}\,\pi(h_3^**h_3)\right)^{1/2}
$$
for all $t \in A^{G}_r$ and $\alpha \in \Delta$. 
Note that  $\sum_\pi \mathrm{tr}\,\pi(h_3^**h_3)<\infty$  since the restriction of the operator $R(h_3)$ to the cuspidal spectrum is of trace class (and hence Hilbert-Schmidt).  This implies the desired rapid decrease of \eqref{is-rd3} and hence the theorem.
\qed

\section{The geometric side} \label{sec-geo-side}

\subsection{Relative Orbital Integrals} \label{ssec-roi}

Let $H$ and $G$ be connected algebraic $F$-groups with $H \leq G \times G$, where $G$ is reductive, and $H$ is the direct product of a reductive and a unipotent group. Let $\chi:H(\A_F) \to \CC^\times$ be a quasi-character trivial on $A_{G,H}H(F)$.  

\begin{defn}
An element $\gamma_v \in G(F_v)$ is \textbf{relevant} if $\chi_v$ is trivial on $H_{\gamma_v}(F_v)$.  An element $\gamma \in G(F)$ is \textbf{relevant} if $\gamma_v$ is relevant for all $v$.
\end{defn}

The point of this definition is that irrelevant elements will not end up contributing to the trace formula.  We note that if $\chi$ is trivial then all elements are relevant.

\begin{defn} Let $v$ be a place of $F$.
For $f_v\in C_c^\infty(G(F_v))$ and $\gamma_v\in G(F_v)$ relevant, unimodular and closed we define the \textbf{local relative orbital integral}:
$$
\textrm{RO}_{\gamma_v}^{\chi_v}(f_v) = \int_{H_{\gamma_v}(F_v)\backslash H(F_v)} \chi_v(h_{\ell},h_r) f_v(h_{\ell}^{-1}\gamma_v h_r) \dot{d}(h_\ell,h_r).
$$
\end{defn}

\begin{rem} The assumption of unimodularity is used to define the right-invariant Radon measure on $H_{\gamma_v}(F_v)\backslash H(F_v)$.
\end{rem}

\begin{prop} \label{prop-roi-conv}
If $\gamma_v\in G(F_v)$ is relevant, unimodular and closed then the integral  $\mathrm{RO}_{\gamma_v}^{\chi_v}(f_v)$ is absolutely convergent.
\end{prop}

\begin{proof}
 Since the measure $\dot{d}(h_\ell,h_r)$ is a Radon measure on $H_{\gamma_v}(F_v) \backslash H(F_v)$, to show the integral is well-defined and absolutely convergent it is enough to construct a  pull-back map
\begin{align} \label{2-maps}
  C_c^\infty(G(F_v))  \lto C_c^\infty(H_{\gamma_v} \backslash H(F_v))
\end{align}
attached to the natural map $H_{\gamma_v} \backslash H(F_v) \lto G(F_v)$.  But this map is a closed embedding
 (since the underlying map of schemes is a closed embedding) and is therefore proper.  Thus 
the pull-back map in \eqref{2-maps} exists.   
\end{proof}

\subsection{Global relative orbital integrals }

\begin{defn} 
For $f\in C_c^\infty(A_G\backslash G(\A_F))$ and relevant, unimodular and closed $\gamma \in G(F)$ we define the \textbf{global relative orbital integral}:
$$
\textrm{RO}_{\gamma}^\chi(f) = \int_{A_{G,H}H_\gamma(\A_F)\backslash H(\A_F)} \chi(h_{\ell},h_r) f(h_{\ell}^{-1}\gamma h_r) \dot{d}(h_\ell,h_r).
$$
\end{defn}

\begin{prop} \label{prop-glob-roi}
If $\gamma\in G(F)$ is relevant unimodular closed then the integral defining
$
\mathrm{RO}_{\gamma}^{\chi}(f)
$
converges absolutely.
\end{prop}

\begin{proof}
As in the proof of Proposition \ref{prop-roi-conv} it suffices to show that the map
\begin{align*} 
H_\gamma \backslash H(\A_F)  \lto G(\A_F)
\end{align*}
is proper, but this is obvious since it is a closed embedding.
\end{proof}

\subsection{The geometric side of the general simple relative trace formula}
Let 
$$
F_{\infty}:=\prod_{v|\infty}F_v
$$
be the product of the archimedian completions of $F$.  
We note that $A \leq H_{\gamma}(F_\infty)$ for all $\gamma \in G(F)$, and 
\begin{align} \label{tama}
\tau(H_\gamma):=\mathrm{meas}_{d(h_\ell, h_r)_{\gamma}}(AH_{\gamma}(F) \backslash H_\gamma(\A_F))
\end{align}
is finite if $\gamma$ is elliptic.  
Let
\begin{align} \label{geo-side}
K_{f}(x,y)=\sum_{\gamma \in G(F)}f(x^{-1} \gamma y).
\end{align}
This kernel is equal to the earlier kernel of \eqref{ker-above} under the additional assumption that $R(f)$ has cuspidal image.  With this in mind, combining Theorem \ref{thm-spectral-side} and the following theorem immediately implies Theorem \ref{thm-stf}:

\begin{thm} \label{thm-geo-side} Assume that 
if the $H(\A_F)$-orbit of $\gamma \in G(F)$ meets the support of $f$ then $\gamma $ is elliptic, unimodular and closed. Then
$$
  \sum_{[\gamma] \in \Gamma(F)} \tau(H_\gamma)\mathrm{RO}_\gamma^\chi(f) = 
\int_{A_{G,H}H(F)\backslash H(\A_F)} \chi(h_{\ell},h_r)K_{f}(h_\ell,h_r)d(h_\ell,h_r).
$$
Moreover, the sum on the left  and the integral on the right are absolutely convergent.
\end{thm}

In the theorem we use the notation $[\gamma]$ for the class of $\gamma$; we will continue to use this convention.  We will also denote by $[\gamma]^{\mathrm{geo}}$ the geometric class of $\gamma$.
To prove Theorem \ref{thm-geo-side}, it is convenient to first prove the following proposition:

\begin{prop} \label{prop-finite-class}
Let $C \subset G(\A_F)$ be a 
compact subset. Then, if $H$ is reductive,
there exist only finitely many closed classes $[\gamma] \in \Gamma(F)$ such that  $H(\A_F)\cdot \gamma' \cap C \neq \emptyset$ for some $\gamma' \in [\gamma]$.
\end{prop}
\noindent Here the $\cdot$ refers to the action \eqref{action}.

We will prove this in several steps.

\begin{lem} \label{lem-1}
Let $C \subset G(\A_F)$ be a 
compact subset. Then, if $H$ is reductive,
there exist only finitely many closed classes $[\gamma]^{\mathrm{geo}} \in \Gamma^{\mathrm{geo}}(F)$ such that $H(\A_F)\cdot \gamma' \cap C \neq \emptyset$ for some $\gamma' \in [\gamma]^{\mathrm{geo}}$.
\end{lem}

\begin{proof}
Since $H$ is reductive there exists a categorical quotient $X$ of $G$ by the action \eqref{action} of $H$; it is an affine scheme of finite type over $F$.  Let 
  $$
  B : G \lto X
  $$ 
be the canonical quotient map.  
Note that if $\gamma,\gamma' \in G(F)$ are closed then $B(\gamma)=B(\gamma')$ if and only if $\gamma$ and $\gamma'$ define the same element of $\Gamma^{\mathrm{geo}}(F)$.  Moreover, assuming $\gamma'$ is closed, if $H(\A_F) \cdot \gamma' \cap C \neq \emptyset$ then $B(C)$ contains the geometric class of $\gamma'$.  On the other hand $B(C) \cap X(F)$ is finite because $B(C)$ is compact and $X(F) \subseteq X(\A_F)$ is discrete and closed.

\end{proof}

We now show that for each closed $\gamma$ there are only finitely many classes in $[\gamma]^{\mathrm{geo}}$ that intersect $C$.  To do this it is convenient to review some Galois cohomology.  

Let $S_0$ be a finite set of places of $F$ including the infinite places.  For $L$ a smooth linear algebraic group over $\OO_F^{S_0}$ let
 $H^1(\A_F,L)$ denote the adelic cohomology of $L$:
\begin{align*}
&H^1(\A_F,L)
:=\left\{(\sigma_v) \in \prod_v H^1(F_v,L): \sigma_v \in
H_{nr}^1(F_v,L) \textrm{ for a.e.~}v \not \in S_0\right\}.
\end{align*}
Here
$$
H_{nr}^1(F_v,L):=\mathrm{Im}\left(H^1(\Gal(F_v^{nr}/F_v),L(\OO_{F_v}^{nr})) \to H^1(F_v,L)\right)
$$
where $F_v^{nr}$ is the maximal unramified extension of $F_v$ and $\OO_{F_v}^{nr}$ is its ring of integers. We endow 
$H^1(F_v,L)$ with the discrete topology for all $v$ and endow
$H^1(\A_F,L)$ with the restricted direct product topology with respect to the subgroups $H^1_{nr}(F_v,L)$ for $v \not \in S_0$ (again given the discrete topology). 
 We have the following lemma:
\begin{lem} \label{lem-proper} The diagonal map $H^1(F,L) \to \prod_vH^1(F_v,L)$ 
has image in $H^1(\A_F,L)$ and the induced map
$$
H^1(F,L) \lto H^1(\A_F,L)
$$
is proper if we give $H^1(F,L)$ the discrete topology.
\end{lem}
Let $S \supseteq S_0$ be a finite set of places of $F$.  It is convenient to say that an element $\sigma=(\sigma_v) \in H^1(\A_F,L)$ is \textbf{unramified outside of $S$} if $\sigma_v \in H_{nr}^1(F_v,L)$ for all $v \not \in S$ and that $\sigma \in H^1(F,L)$ is \textbf{unramified outside of $S$} if $\sigma$ maps to an element of $H^1(\A_F,L)$ unramified outside of $S$ under the diagonal map (i.e. the map of Lemma \ref{lem-proper}).
 
\begin{proof} It is not hard to see that $H^1(F,L)$ has image in $H^1(\A_F,L)$. We now prove the properness statement.  For this we follow the proof of \cite[Proposition 4.4]{HS}.
Since $H^1(F_v,L)$ is finite for all $v$ it is enough to show that for all sufficiently large $S \supseteq S_0$, 
the inverse image of $\prod_{v \not \in S}H^1_{nr}(F_v,L)$ in $H^1(F,L)$ is finite, in other words, there are only finitely many classes in $H^1(F,L)$ unramified outside of $S$.
We denote by $L^\circ$ the schematic closure in $L$ of the neutral component of $L_F$.  
By enlarging $S$ if necessary we can assume that 
$L$, $L^\circ$, $\pi_0(L):=L/L^\circ$, and $\mathrm{Aut}(\pi_0(L))$ are all smooth over $\OO_{F}^S$ and that the sequence
$$
1 \lto L^\circ \lto L \lto \pi_0(L) \lto 1
$$
is exact, which in turn yields a cartesian diagram with exact rows:
\begin{align}
\label{exact-1}
\begin{CD}
H^1(\Gal(F^{nr}_v/F_v),L^\circ(\OO_{F_v}^{nr})) @>>> H^1(\Gal(F^{nr}_v/F_v),L(\OO_{F_v}^{nr})) @>>{\alpha}> H^1(\Gal(F^{nr}_v/F_v),\pi_0(L)(\OO_{F_v}^{nr}))\\
@VVV @VVV @VV{\beta}V\\
H^1(F_v,L^{\circ}) @>>> H^1(F_v,L) @>>> H^1(F_v,\pi_0(L))
\end{CD}
\end{align}
for all $v \not \in S$.
All of the maps are the natural ones; we have just labeled two of them $\alpha$ and $\beta$.  We now use this diagram to prove that 
the map 
\begin{align} \label{bij}
H^1_{nr}(F_v,L) \lto H^1_{nr}(F_v,\pi_0(L))
\end{align}
is injective.  

We first claim that $H^1(\Gal(F^{nr}_v/F_v),L^\circ(\OO_{F_v}^{nr}))$ is trivial for all $v \not \in S_0$.
  Indeed, let $X$ be an $L^{\circ}_{\OO_{F_v}}$-torsor representing an element.  Then, denoting by $\varpi_v$ a uniformizer for $\OO_{F_v}$ one has
$$
X(\OO_{F_v}/\varpi_v) \neq \emptyset
$$
by Lang's theorem \cite[\S III.2.3]{SerreGC}.  Since $X$ is smooth, Hensel's lemma implies that the map $X(\OO_{F_v})\to X(\OO_{F_v}/\varpi_v)$ is surjective, so in particular 
$X(\OO_{F_v}) \neq \emptyset$, proving our claim.  
This implies that the map $\alpha$ in \eqref{exact-1} is injective.  

We now claim that the map 
\begin{align} \label{inj}
\beta: H^1(\Gal(F^{nr}_v/F_v),\pi_0(L)(\OO_{F_v}^{nr})) \lto H^1(F_v,\pi_0(L)(F_v)),
\end{align}
of \eqref{exact-1} is injective.  Assuming this it follows that \eqref{bij} is injective as asserted above.  
To prove that $\beta$ is injective, let $X_1,X_2$ be two $\pi_0(L)_{\OO_{F_v}}$-torsors isomorphic over $\OO_{F_v}^{nr}$ such that $X_{1F_v} \cong X_{2F_v}$; which is to say that the classes of these torsors map to the same element of $H^1(F_v,\pi_0(L)(F_v))$ under $\beta$.  
The $\OO_{F_v}^{nr}$-isomorphisms between $X_{1\OO_{F_v}^{nr}}$ and $X_{2\OO_{F_v}^{nr}}$ form an $\mathrm{Aut}(\pi_0(L))_{\OO_{F_v}}$-torsor $Y$ such that $Y(F_v) \neq \emptyset$ (since $X_{1F_v} \cong X_{2F_v}$), and $Y(\OO_{F_v}) \neq \emptyset$ if and only if $X_1 \cong X_2$ (over 
$\OO_{F_v}$), i.e. if and only if $X_1$ and $X_2$ represent the same class in $H^1(\Gal(F^{nr}_v/F_v),\pi_0(L)(\OO_{F_v}^{nr}))$.  But $\mathrm{Aut}(\pi_0(L))$ is proper over $\OO_{F_v}$ (even finite), and hence so is $Y$. By the valuative criterion of properness $Y(F_v) \neq \emptyset$ implies $Y(\OO_{F_v}) \neq \emptyset$, implying that $X_1 \cong X_2$ (over $\OO_{F_v}$).  As already remarked, this completes our proof that \eqref{bij} is injective as asserted above.

Suppose that $\sigma \in H^1(F,L)$ is unramified outside of $S$.  Then the image of $\sigma$ in 
$$
\mathrm{Im}\left(H^1(F,\pi_0(L)) \to H^1(\A_F,\pi_0(L))\right)
$$ 
say $\xi$, is also unramified outside of $S$.  The cocycle $\xi$ is attached to the spectrum of an \'etale $F$-algebra (i.e. direct sum of finite extension fields) of degree at most $\pi_0(L)(\bar{F})$ that is unramified outside of $S$.  There are only finitely many such \'etale $F$-algebras, so to complete the proof of the lemma it suffices to fix a cocycle $\xi$ and show that there are only finitely many $\sigma \in H^1(F,L)$ unramified outside of $S$ that map to it.  For this, we combine the fact that $H^1(F_v,L)$ is finite for all $v$ and the injection \eqref{bij} to conclude that there are only finitely many elements of $H^1(\A_F,L)$ unramified outside of $S$ that map to $\xi$.  We now employ the Borel-Serre theorem \cite[\S III.4.6]{SerreGC} which states that the fibers of the diagonal map $H^1(F,L) \to \prod_{v}H^1(F_v,L)$ are finite, to deduce that there are only finitely many $\sigma \in H^1(F,L)$ mapping to $\xi$ that are unramified outside of $S$ and thereby complete the proof of the lemma.
\end{proof}

Now assume that $L \leq M$ are smooth linear algebraic groups over $\OO_F^S$ such that $M$ has connected fibers.  Then the map $M \lto L \backslash M$ is smooth and surjective.  We obtain a characteristic map
$$
cl:L\backslash M(\A_F) \lto H^1(\A_F,L).
$$

\begin{lem} \label{lem-cl-comp} The characteristic map $cl$ maps compact sets to compact sets.
\end{lem}

\begin{rem} We do not know whether $cl$ is continuous.
\end{rem}

\begin{proof}
Any cocycle $\sigma \in cl(L \backslash M(F_v)) \subseteq H^1(F_v,L)$ gives rise to forms ${}_\sigma L,{}_\sigma M$ of $L_{F_v}$ and $M_{F_v}$ equipped with a map
\begin{align} \label{twist-map}
{}_\sigma L(F_v) \backslash {}_\sigma M(F_v) \lto L \backslash M(F_v)
\end{align}
with the property that the inverse image of $\sigma$ under $\mathrm{cl}$ is the image of \eqref{twist-map}
 (compare \cite[\S I.5.4, Corollary 2]{SerreGC}). Moreover, ${}_\sigma M(F_v) \to L \backslash M(F_v)$ is open (see above the proof of \cite[Theorem 4.5]{Conrad}).  Thus the maps $cl:L \backslash M(F_v) \to H^1(F_v,L)$ are continuous for each $v$ if we give $H^1(F_v,L)$ the discrete topology.

The map $M(\OO^{nr}_{F_v}) \to L \backslash M(\OO^{nr}_{F_v})$ is surjective by Hensel's lemma, and it follows that 
 $cl(L \backslash M(\OO_{F_v})) \subseteq H_{nr}^1(F_v,L)$, which completes the proof of the lemma.
\end{proof}

We now prove Proposition \ref{prop-finite-class}:

\begin{proof}[Proof of Proposition \ref{prop-finite-class}]

For a large enough set $S_0$ of places of $F$ including the infinite places we can and do choose models of $H_\gamma \leq H$ over $\OO_F^{S_0}$ that are smooth linear algebraic groups.  We use the same letters to denote these models and use the models to define adelic cohomology as above.  

In view of Lemma \ref{lem-1} it suffices to check that for a given closed $\gamma \in G(F)$ there are finitely many $\gamma'$ in the geometric class of $\gamma$ such that $H(\A_F) \cdot \gamma' \cap C \neq \emptyset$.

One has a commutative diagram with exact rows
$$
\begin{CD}
H_\gamma(F) @>>>  H(F) @> >>H_\gamma \backslash H(F) @>{cl}>> H^1(F,H_\gamma)\\
@VVV	@VVV @VVV @V{a}VV\\
H_{\gamma}(\A_F) @>>> H(\A_F) @>>> H_\gamma \backslash H(\A_F) @>{cl}>>  H^1(\A_F,H_{\gamma}) 
\end{CD}
$$
and the image of the map $cl$ on the upper line can be identified with the set of classes in the geometric class of $\gamma$.
We give the first three sets on the bottom row their natural topologies and give $H^1(\A_F,H_{\gamma})$ the topology described above Lemma \ref{lem-proper}.

Identifying $H_\gamma \backslash H(\A_F)$ with a subset of $G(\A_F)$ via the action of $H(\A_F)$ on $\gamma$, the set of $\gamma'$ in the geometric class of $\gamma$ such that $H(\A_F) \cdot \gamma' \cap C \neq \emptyset$ 
injects into the subset of $cl(H_\gamma \backslash H(F))$ mapping to 
\begin{align} \label{is-compact}
cl(C \cap H_\gamma \backslash H(\A_F))
\end{align}
under $a$.  Since $a$ is proper by 
Lemma \ref{lem-proper}, it suffices to show \eqref{is-compact} is compact.  Since $C \cap H_\gamma \backslash H(\A_F)$ is compact by the fact $\gamma$ is closed, the compactness of \eqref{is-compact} follows from Lemma \ref{lem-cl-comp}.
\end{proof}

\begin{rem} One can prove Proposition \ref{prop-finite-class} in a simpler manner as follows: Let $C \subset G(\A_F)$ be a compact set.  Observe that 
the $\gamma' \in G(F)$ in the geometric class of a given closed $\gamma \in G(F)$ such that $H(\A_F) \cdot \gamma' \cap C \neq \emptyset $
 are in the intersection of $C$ and the image of the topological embeddings  
$$
H_\gamma \backslash H(F) \lto H_{\gamma} \backslash H(\A_F) \lto G(\A_F).
$$
Since $H_{\gamma} \backslash H(\A_F) \cap C$ is compact and $H_\gamma \backslash H(F)$ is discrete and closed in  $H_{\gamma} \backslash H(\A_F)$ we can deduce Proposition \ref{prop-finite-class} from Lemma \ref{lem-1}.  However, the more refined information presented in the discussion above ought to be useful as a starting point towards future work on the stabilization of the relative trace formula.

\end{rem}

We now prove the theorem:
\begin{proof}[Proof of Theorem \ref{thm-geo-side}]
Proceeding formally for the moment, we have
\begin{align} \label{rel-sum}
\sum_{\substack{[\gamma] \in \Gamma(F)\\ \gamma \textrm{ relevant}}} &\tau(H_\gamma)\mathrm{RO}_\gamma^\chi(f)\\
 &= \sum_{\substack{[\gamma] \in \Gamma(F)\\ \gamma \textrm{ relevant}}} \tau(H_\gamma)\int_{(A \backslash A_{G,H}) H_\gamma(\A_F)\backslash H(\A_F)} \chi(h_\ell, h_r) f(h_\ell^{-1}\gamma h_r) \dot{d}(h_\ell,h_r). \nonumber 
\end{align}
Notice that
$$
  \int_{ A_{G,H} H_\gamma (F) \backslash H(\A_F)} \chi(h_\ell,h_r)f(h_\ell^{-1}\gamma h_r)d(h_\ell,h_r) = 0
$$
if $\gamma$ is not relevant, because in this case
$$
 \int_{A  H_\gamma(F)\backslash H_\gamma(\A_F)} \chi(h_{\ell},h_r) d(h_{\ell},h_r)_{\gamma} = 0.
$$
Thus \eqref{rel-sum} is equal to 
\begin{align*}
   \sum_{[\gamma] \in \Gamma(F)}& \int_{A_{G,H} H_\gamma(F) \backslash H(\A_F)} \chi(h_\ell, h_r)f(h_\ell^{-1}\gamma h_r)d(h_\ell, h_r)\\ 
&= \int_{ A_{G,H}  H(F)\backslash  H(\A_F)} \chi(h_\ell, h_r^{-1}) \sum_{\gamma \in G(F)} f(h_\ell^{-1}\gamma h_r)d(h_\ell,h_r)\\
&= \int_{  A_{G,H}  H(F)\backslash  H(\A_F)} \chi(h_\ell,h_r) K_{f}(h_\ell,h_r) d(h_\ell, h_r). 
\end{align*}

We now justify these formal manipulations.  By dominated convergence, it suffices to consider the case where $\chi=|\chi|$ and $f$ is nonnegative; we henceforth assume this.  Suppose that $\gamma \in G(F)$ is relevant, unimodular and closed.  Then by Proposition \ref{prop-glob-roi} one has
$$
|\mathrm{RO}_{\gamma}^\chi(f)| < \infty.
$$
If $\gamma$ is unimodular, closed and elliptic we have
$$
|\tau(H_{\gamma})| < \infty.
$$
If, in addition, $H$ is reductive then the sum over $\gamma$ in \eqref{rel-sum} is finite by 
Proposition \ref{prop-finite-class} so in this case our formal manipulations are justified.  

In the general case, write
$$
H=M_H \times N_H
$$
where $M_H$ (resp.~$N_H$) is reductive (resp.~unipotent).

Decompose the measure $d(h_{\ell},h_r)$ on $A_{G,H}H(F) \backslash H(\A_F)$ as $d(m_\ell,m_r)d(n_\ell,n_r)$ where $d(m_\ell,m_r)$ (resp.~$d(n_\ell,n_r)$) is a measure on $A_{G,H}M_H(F) \backslash M_H(\A_F)$ (resp.~$N_H(F) \backslash N_H(\A_F)$) induced by a Haar measure on $A_{G,H} \backslash M_H(\A_F)$ (resp.~$N_H(\A_F)$).
Since $N_H(F) \backslash N_H(\A_F)$ is compact, we can choose a compact subset
 $\Omega \subseteq N(\A_F)$ such that 
\begin{align*}
\int_{ A_{G,H} H(F) \backslash H(\A_F)}&|\chi|(h_\ell,h_r)K_{f}(h_\ell,h_r)d(h_{\ell},h_r)\\
&=\int_{A_{G,H} M_H(F) \backslash M_H(\A_F)\times \Omega}|\chi|(m_\ell n_\ell,m_rn_r)K_{f}(m_\ell n_\ell,m_rn_r)d(m_{\ell},m_r)d(n_\ell,n_r)\\
&=\int_{ A_{G,H}M_H(F) \backslash M_H(\A_F)}|\chi|(m_\ell,m_r)K_{\widetilde{f}}(m_\ell ,m_r)d(m_{\ell},m_r)
\end{align*}
where
$$
\widetilde{f}(x):=\int_{\Omega}|\chi|(n_\ell,n_r)f(n_\ell^{-1}xn_r)d(n_\ell,n_r) \in C_c^\infty(A \backslash G(\A_F)).
$$
This allows us to reduce to the reductive case with which we have already dealt.
\end{proof}

\section{A relative Weyl law}
\label{sec-Weyl}

Let $G$ be a split adjoint semisimple group over $\QQ$. Note that $G(\QQ) \backslash G(\A_{\QQ})$ is of finite volume but non-compact. We also let $G$ denote the Chevalley group over $\ZZ$ whose generic fiber is $G$.  Fix a maximal compact subgroup $K_\infty \leq G(\RR)$ and a compact open subgroup $K^\infty \leq G(\A_{\QQ}^\infty)$ and let
$$
K:=K_\infty \times K^\infty.
$$
We assume that $K^S=G(\widehat{\ZZ}^S)$ for any sufficiently large finite set of places $S$ of $\QQ$ containing infinity.  For our later use we fix a maximal split torus $T \leq G$ and assume that the Cartan involution fixing $K_\infty$ acts as inversion on the identity component $T(\RR)^+$ of $T(\RR)$ in the real topology.
We impose the following torsion-freeness assumption:
\begin{itemize}
\item[(TF)] For all $g \in G(\A_\QQ^\infty)$ the group $g^{-1} K^\infty g \cap G(\QQ)$ is torsion-free.
\end{itemize}
This can always be arranged by taking $K^\infty$ to be contained in a sufficiently small principal congruence subgroup.

To deduce the relative Weyl law of Theorem \ref{thm-Weyl-intro}, we investigate the following special case of the setting of the previous sections of the paper:  

Let $H_0 \leq G$ be a subgroup that is a direct product of a reductive group and a unipotent group and let $H \leq G \times G$ be the image of the diagonal embedding $H_0 \hookrightarrow G \times G$. 
We point out that though $H_0(\QQ) \backslash H_0(\A_{\QQ})$ is compact, we make no such assumption on $G(\QQ) \backslash G(\A_{\QQ})$, so Theorem \ref{thm-Weyl-intro} does not follow in any obvious way from the usual Weyl law and its local variants.  Moreover, we will also show in Proposition \ref{upper} how the same asymptotic would follow for non-compact $H_0(\QQ) \backslash H_0(\A_{\QQ})$ of finite volume provided that we knew the upper bound of the relative Weyl law (in the setting of the usual Weyl law this was proven in \cite{Donnelly}).

We restate Theorem \ref{thm-Weyl-intro} for convenience:

\begin{thm}\label{thm-Weyl-main}  Assume that $H_0(\QQ) \backslash H_0(\A_{\QQ})$ is compact.
As $X \to \infty$
one has
\begin{align}
\sum_{\pi:\,\pi(\Delta) \leq X}\sum_{\varphi \in \mathcal{B}(\pi)^K}\int_{H_0(\QQ) \backslash H_0(\A_{\QQ})}|\varphi(h)|^2dh \sim \alpha(G)\mathrm{meas}_{dh}(H_0(\QQ) \backslash H_0(\A_{\QQ}))X^{d/2},
\end{align}
where the sum is over isomorphism classes of cuspidal automorphic representations $\pi$ of $G(\A_\QQ)$, $\mathcal{B}(\pi)$ is an orthonormal basis of $\pi$-isotypic subspace of $L^2_{\mathrm{cusp}}(G(\QQ) \backslash G(\A_{\QQ}))$, $\pi(\Delta)$ is the eigenvalue of the Casimir operator $\Delta$ acting on the space of $K_\infty$-fixed vectors in $\pi$, and $d=\dim (G(\RR)/K_{\infty})$. 
\end{thm}
\noindent Here $\alpha(G)>0$ is the same constant appearing in \cite{LV}, and the Casimir operator and the Haar measure on $G(\RR)$ are normalized as in \cite{LV}.  The Haar measure on $G(\A_\QQ^\infty)$ is normalized to give
$K^\infty$ volume $1$.

The proof of Theorem \ref{thm-Weyl-main} follows from the observation
that if we replace the diagonal embedding $G \hookrightarrow G \times G$ considered in Lindenstrass and Venkatesh's work \cite{LV} by the diagonal embedding $H_0 \hookrightarrow G \times G$, the argument of \cite{LV} can be followed line by line to deduce the result.  In particular, one can use the same test functions that were constructed in loc.~cit.  We will give a few more details but will be quite brief.

With a view towards future generalizations, until otherwise stated we merely assume  that $H_0(\QQ) \backslash H_0(\A_{\QQ})$ has finite volume (which is not implied by the fact that $G(\QQ) \backslash G(\A_{\QQ})$ has finite volume).  

Arguing exactly as in \cite{LV} one proves the following theorem:


\begin{prop}\label{upper}
Let $H_0(\QQ) \backslash H_0(\A_{\QQ})$ be of finite volume (not necessarily compact) and let $0 < \varepsilon <1$. If we assume the upper bound of the relative Weyl law, namely, that for $X\rightarrow \infty,$ one has
$$
\sum_{\pi:\,\pi(\Delta) \leq X}\sum_{\varphi \in \mathcal{B}(\pi)^K}\int_{H_0(\QQ) \backslash H_0(\A_{\QQ})}|\varphi(h)|^2dh \leq (\alpha(G)+O(\varepsilon))\mathrm{meas}_{dh}(H_0(\QQ) \backslash H_0(\A_{\QQ}))X^{d/2},
$$
then Theorem \ref{thm-Weyl-main} follows. \qed
\end{prop}

In \cite{LV}, the upper bound of Proposition \ref{upper} follows from work of Donnelly \cite{Donnelly}.  Interestingly, the corresponding relative analogue is not known.  However, in case where $H_0(F) \backslash H_0(\A_F)$ is compact one can establish the following result using standard techniques:

\begin{prop}\label{prop-upperbound}  Suppose that  $H_0(\QQ) \backslash H_0(\A_{\QQ})$ is compact and that $0 < \varepsilon<1$.  With notation as in Theorem \ref{thm-Weyl-main}, for $X \in \RR_{>0}$ one has the upper bound:
$$
\sum_{\pi:\,\pi(\Delta) \leq X}\sum_{\varphi \in \mathcal{B}(\pi)^K}\int_{H_0(\QQ) \backslash H_0(\A_{\QQ})}|\varphi(h)|^2dh \leq (\alpha(G)+O(\varepsilon))\mathrm{meas}_{dh}(H_0(\QQ) \backslash H_0(\A_{\QQ}))X^{d/2}.
$$

\end{prop}

\begin{proof}
One can mimic the argument in \cite[\S 5]{LV}.   There are only two minor differences between the argument in loc.~cit. and the argument proving the proposition above.  First, in \cite[Lemma 2(4)]{LV} one replaces $1-\varepsilon$ with $1+\varepsilon$, since we are interested in upper bounds.  Second, one has to include Eisenstein series in the expansion of the spectral kernel. However, unlike in the usual trace formula, their contribution is absolutely convergent in the setting above because we have assumed $H_0(\QQ) \backslash H_0(\A_{\QQ})$ is compact.  This contribution is also positive by the choice of test function in loc.~cit.  
\end{proof}

Combining Proposition \ref{prop-upperbound} and  Proposition \ref{upper} yields Theorem \ref{thm-Weyl-main}.


\end{document}